\documentclass[12pt]{amsart}
\usepackage{graphicx}
\usepackage{amssymb}
\usepackage{epstopdf}
\usepackage{hyperref}
\usepackage{morefloats}

\newtheorem{lemma}{Lemma}
\newtheorem{theorem}{Theorem}

\begin{document}

\title{Constructing and Counting Hexaflexagons}
\author{Marshall Hampton}

\begin{abstract}
Foldable paper constructions known as flexagons have been studied since 1939.  In this paper we review the construction of hexagonal flexagons (hexaflexagons) and compute the number of distinct hexaflexagons with $n$ faces.  
\end{abstract}

\maketitle

\section{Introduction}

The study of the curious paper folded objects known as flexagons began in 1939 with the discovery by Arthur Stone of the trihexaflexagon. As Martin Gardner relates in his 1956 article \cite{Gardner} (his first column in Scientific American!), Stone's discovery was investigated further by an all-star recreational math team: Richard Feynman, John W. Tukey, and Bryant Tuckerman, the `Flexagon Committee'.  Academic papers on flexagons began to appear in the 1950s \cite{Maunsell1954, OakleyWisner, Crampin1957, Wheeler1958} as well as books and other educational materials \cite{Jones1966}.

The combinatorics of flexagons are interesting, with connections to enumeration and generation of related combinatorial structures and computer science \cite{Callan2012, Karim}.  In this article we focus on hexaflexagons, described below, but there is a wide variety of other interesting flexagons which remain to be more fully explored mathematically \cite{Pook2009}.

\section{Constructing hexaflexagons}

A general procedure for constructing hexaflexagons was developed by Stone and the Flexagon Committee.  Unfortunately they did not publish their findings but the procedure has been described by others, in slightly varying ways \cite{Conrad, Jackson, Nishi2008}.  We will illustrate this construction from an iterative viewpoint, going from the basic trihexflexagon up to the three hexahexaflexagons.

Consider first the trihexaflexagon, shown below with labels for the top and bottom sides, respectively, of each triangle.

\begin{center}
\includegraphics[width=3.8in]{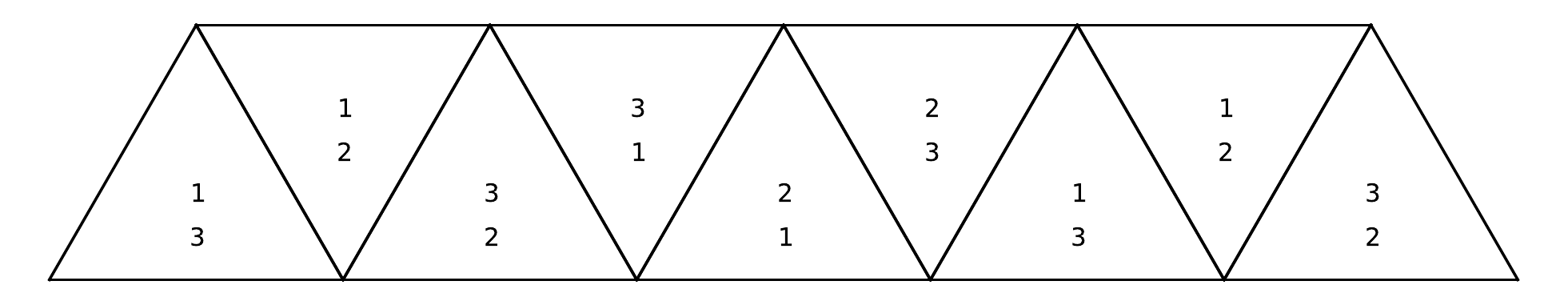}
\end{center}

The flexagon is folded from the pattern by folding adjacent faces together if they have the same label.  For the trihexaflexagon above, the top faces can all be folded in pairs, and then two more folds are required to reach the final hexagonal form (then the beginning and ending edges must be taped together; alternatively an extra triangle can be added and then two faces are glued together).

We can generate the top and bottom labels from a path in what we will call the {\it pattern} of the hexaflexagon, which keeps track of the orientation and location of the required folds:

\begin{center}
\includegraphics[width=2in]{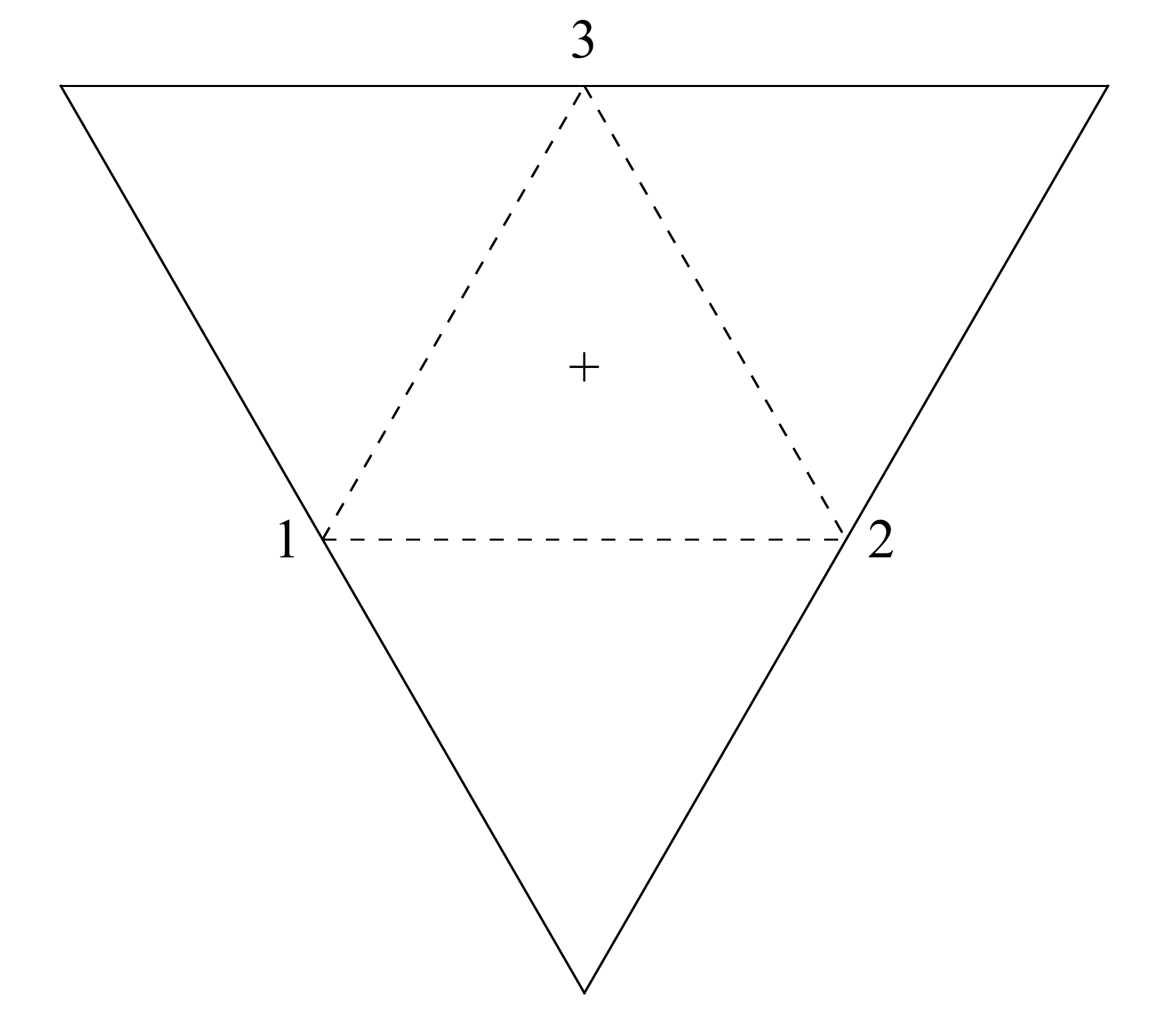}
\end{center}

The outer triangle is labeled counter-clockwise from $1$ to $n$, in this case $n=3$.
We begin at the point labeled $1$ on the inscribed triangle,  and then proceed to travel on the inscribed triangle, starting counter-clockwise.  The inscribed triangle has an orientation, here indicated by ``$+$''.  Recording both the vertices and the signs in order, we get the label sequence $(1,2,3)$ and the signs $(+,+,+)$.  We will represent the sign sequence by $1$s and $-1$s, so here it would be $(1,1,1)$.  

To obtain the labeled hexaflexagon, we begin with a triangle, and continue in one direction for as long as the sign sequence does not change. We repeat this three times.  For the trihexaflexagon, all of the signs are the same, so we end up with a straight line of $9$ triangles.  We obtain the top and bottom labels by the following procedure:
\begin{enumerate}
\item Use the label sequence to alternately give a top and then bottom label:
\begin{verbatim}
1   3
  2  
\end{verbatim}

\item If $n$ is odd, as it here ($n$=3), we extend this pattern twice by alternating between starting with the bottom and top labels:
\begin{verbatim}
1   3   2   1   3
  2   1   3   2  
\end{verbatim}

For even $n$ the pattern is repeated always starting with the top label.

\item The remaining labels are filled in with values one less than the other one for that triangle; with $0$ replaced by $n$:
\begin{verbatim}
1 1 3 3 2 2 1 1 3
3 2 2 1 1 3 3 2 2
\end{verbatim}

\end{enumerate}

We can add another set of six triangles - an additional hexaflexagon face - to the hexaflexagon by extending the pattern with another triangle.  The added triangle has an orientation opposite to that of the one it is being added to.

\begin{center}
\includegraphics[width=2in]{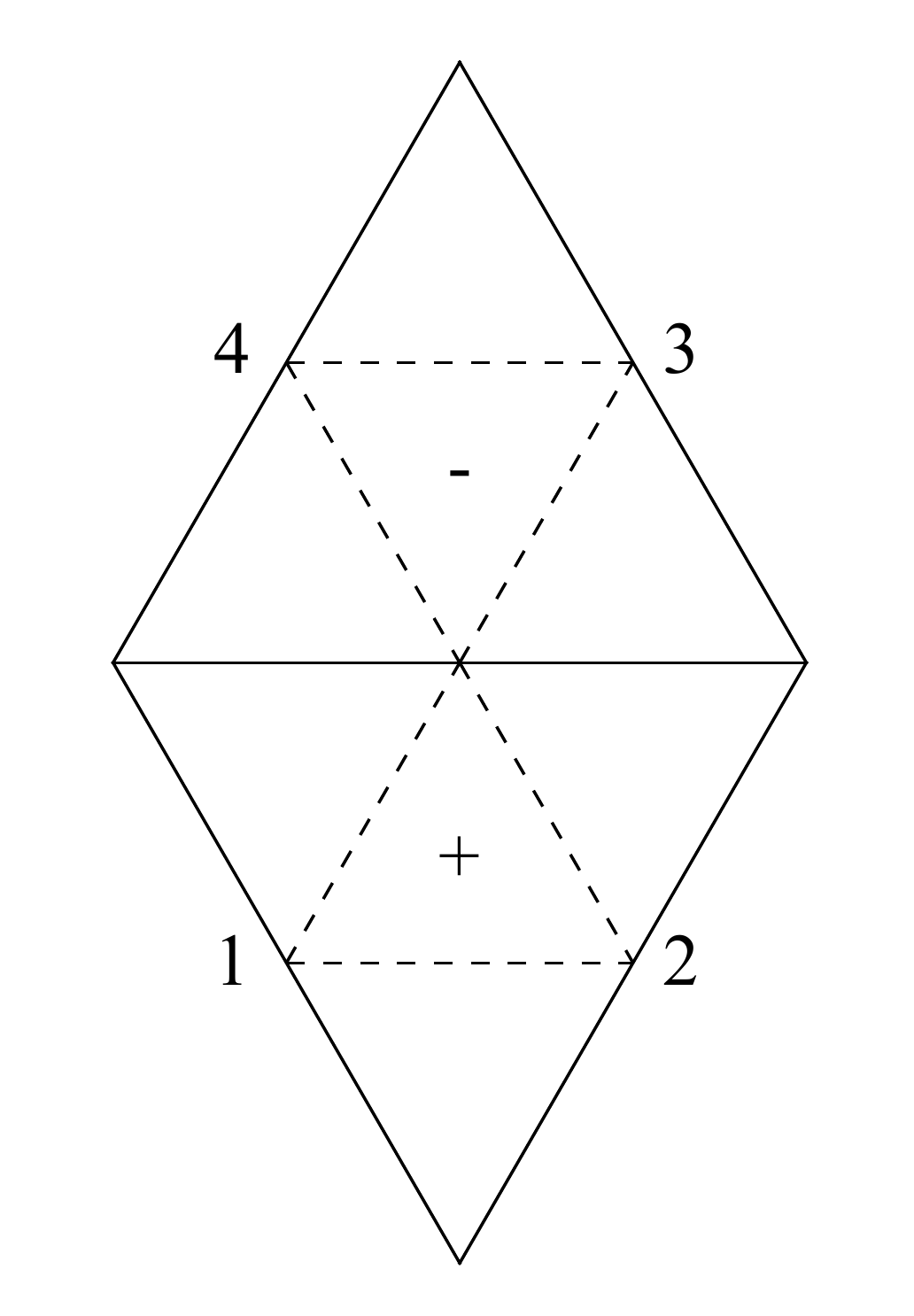}
\end{center}

Again tracing through the dashed inscribed path, we get the label sequence $(1,2,4,3)$, which gives us the following signs and  top and bottom labels:

$$\begin{array}{cc} 
\text{signs: } & (1,1,-1,-1) \\ 
\text{top: } & (1, 1, 4, 2, 1, 1, 4, 2, 1, 1, 4, 2, 1) \\
\text{bottom: } & (4, 2, 3, 3, 4, 2, 3, 3, 4, 2, 3, 3, 4) \end{array}$$

Following same procedure with the signs, we get the only tetrahexaflexagon:

\begin{center}
\includegraphics[width=3.5in]{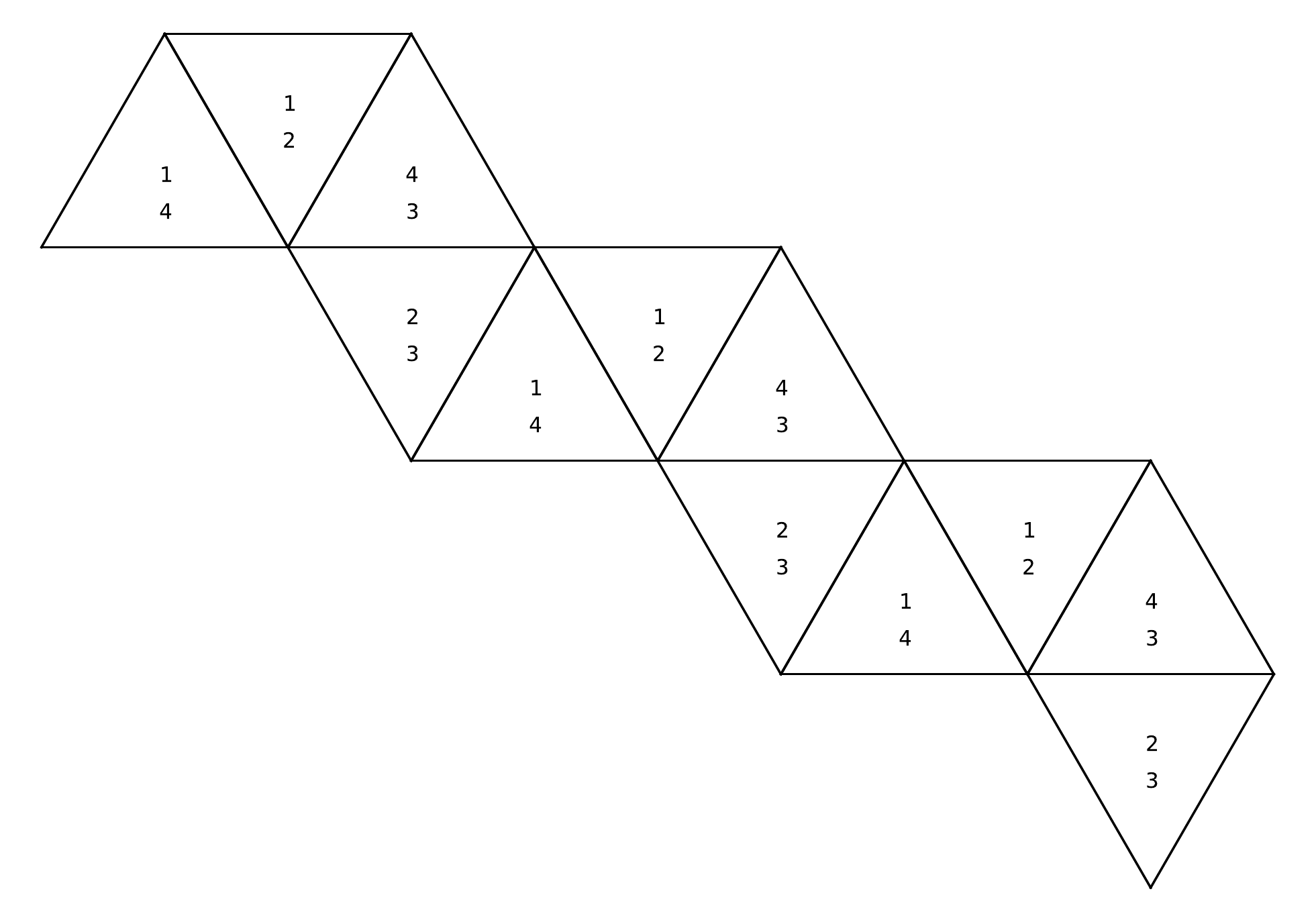}
\end{center}

From the tetrahexaflexagon pattern, it initially seems that we can extend the pattern in a number of ways.  However all of these choices will give rise to the same quintahexaflexagon and diagram up to cyclicly permuting the labels, reversing the labels, and reversing the signs.  If we choose the representative pattern with the most positive signs, ordered counterclockwise, we have the following pattern:

\begin{center}
\includegraphics[width=2.25in]{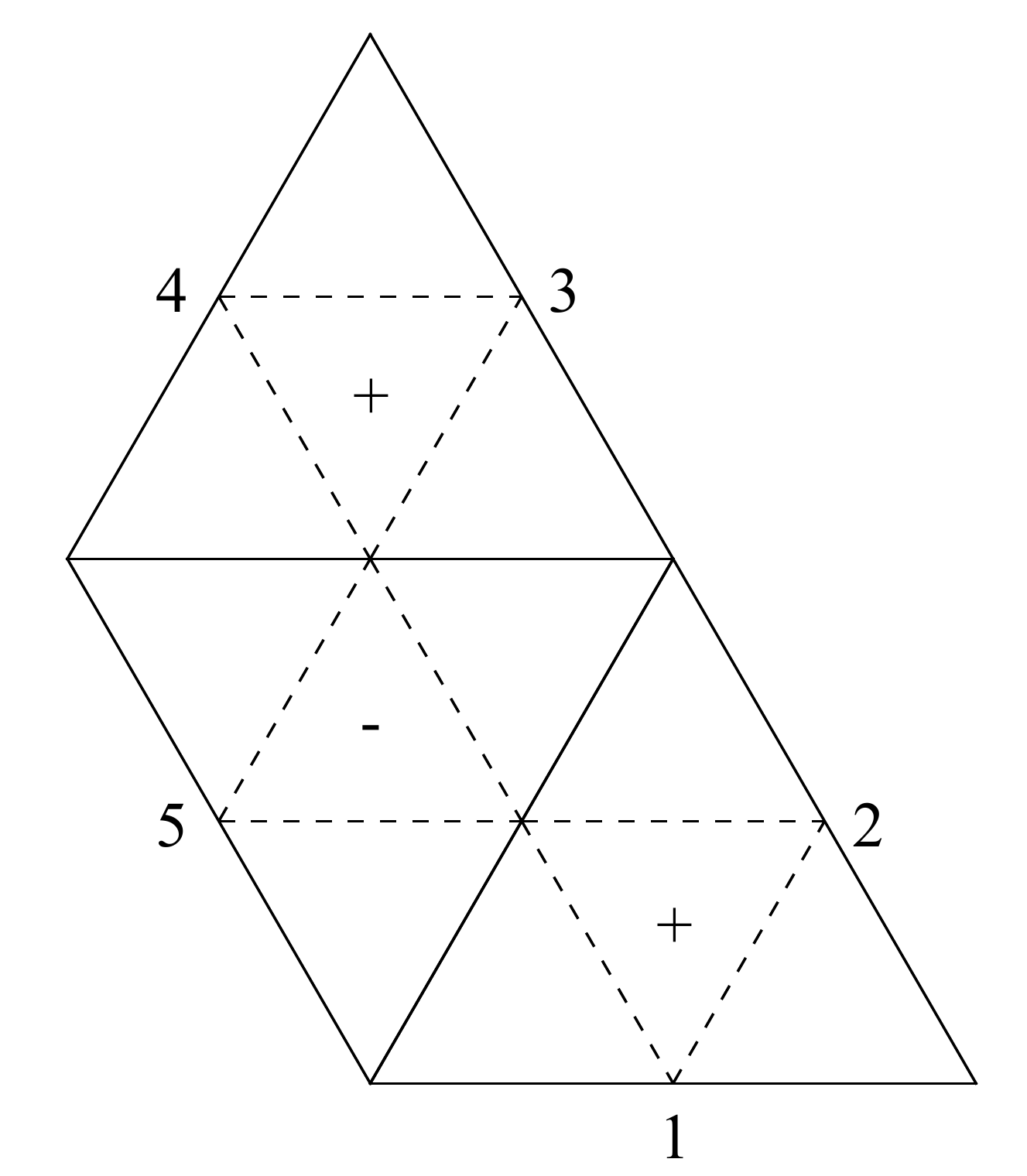}
\end{center}

Again following our procedure we generate the following representative of the quintahexaflexagon:

\begin{center}
\includegraphics[width=3in]{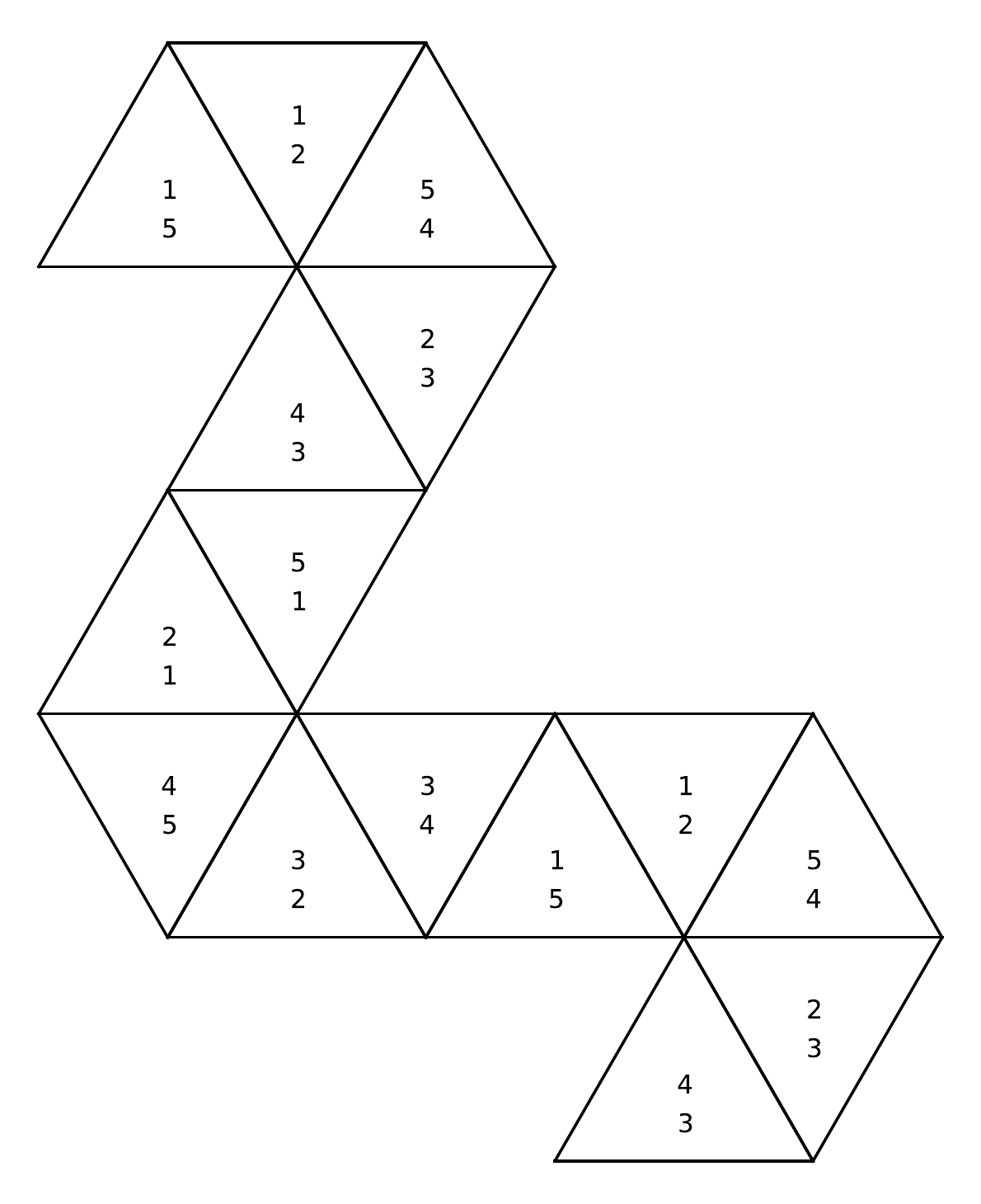}
\end{center}

Finally we will illustrate what happens for six faces.  Here there are three distinct patterns and hexahexaflexagons.  The first is the classic hexahexaflexagon, with a sign sequence of $(1,1,1,1,1,1)$.

\begin{center}
\includegraphics[width=2.5in]{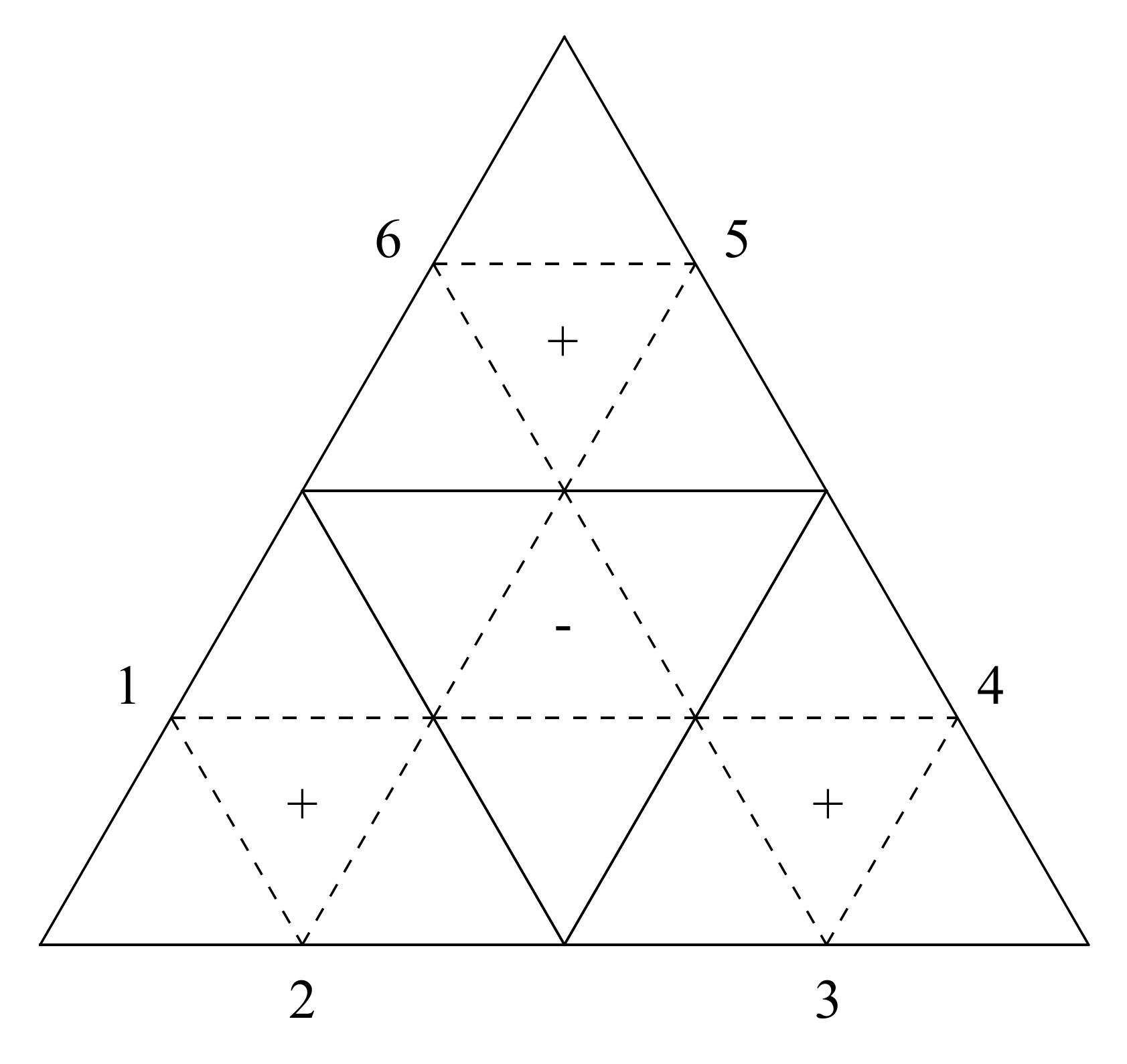}
\end{center}

\begin{center}
\includegraphics[width=5in]{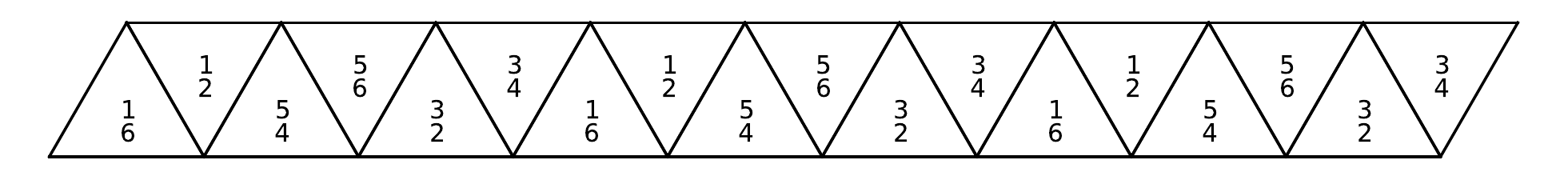}
\end{center}

The second hexahexaflexagon has a sign sequence of $(1,1,1,-1,-1,-1)$:

\begin{center}
\includegraphics[width=2in]{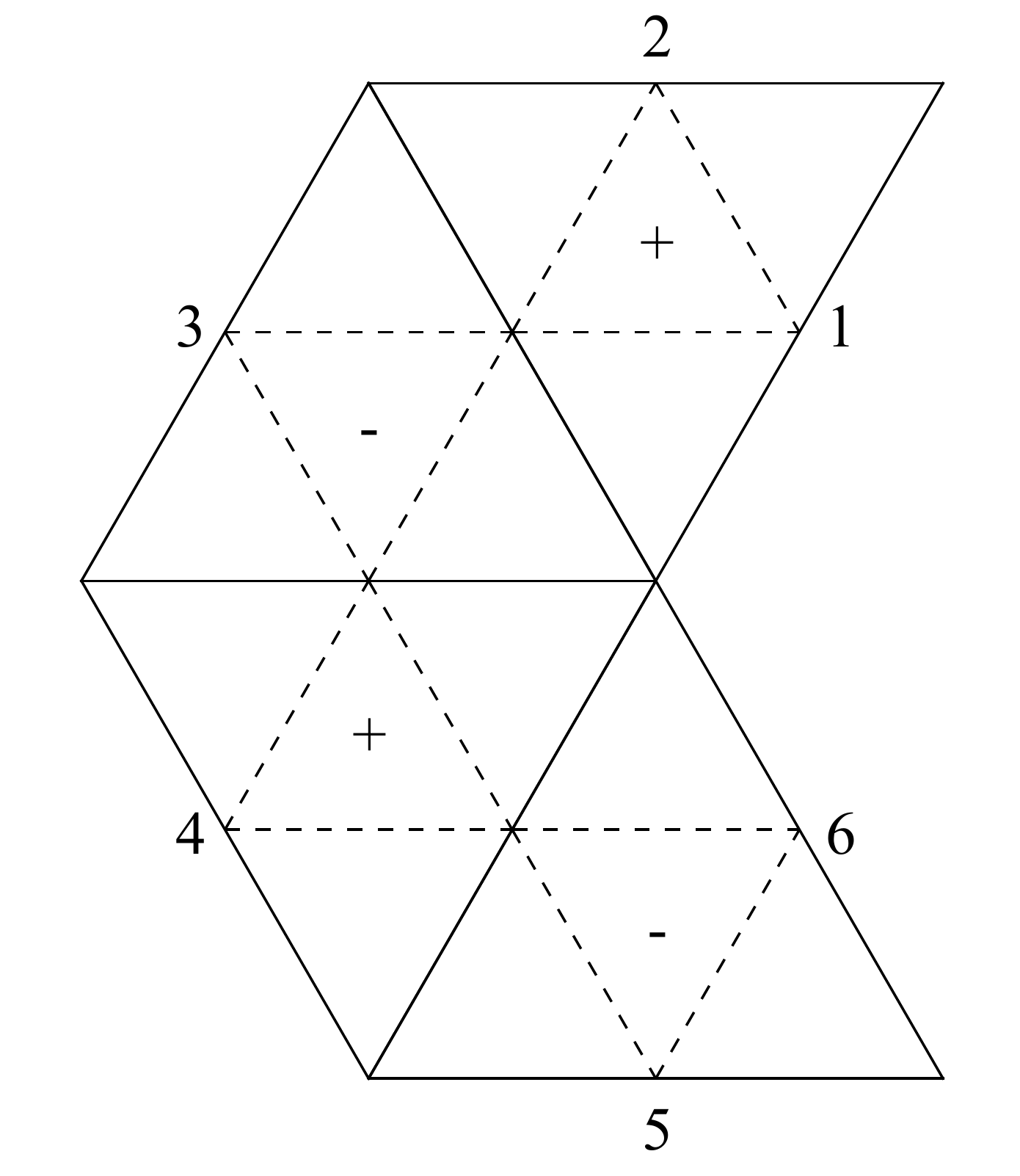}
\end{center}

\begin{center}
\includegraphics[width=3in]{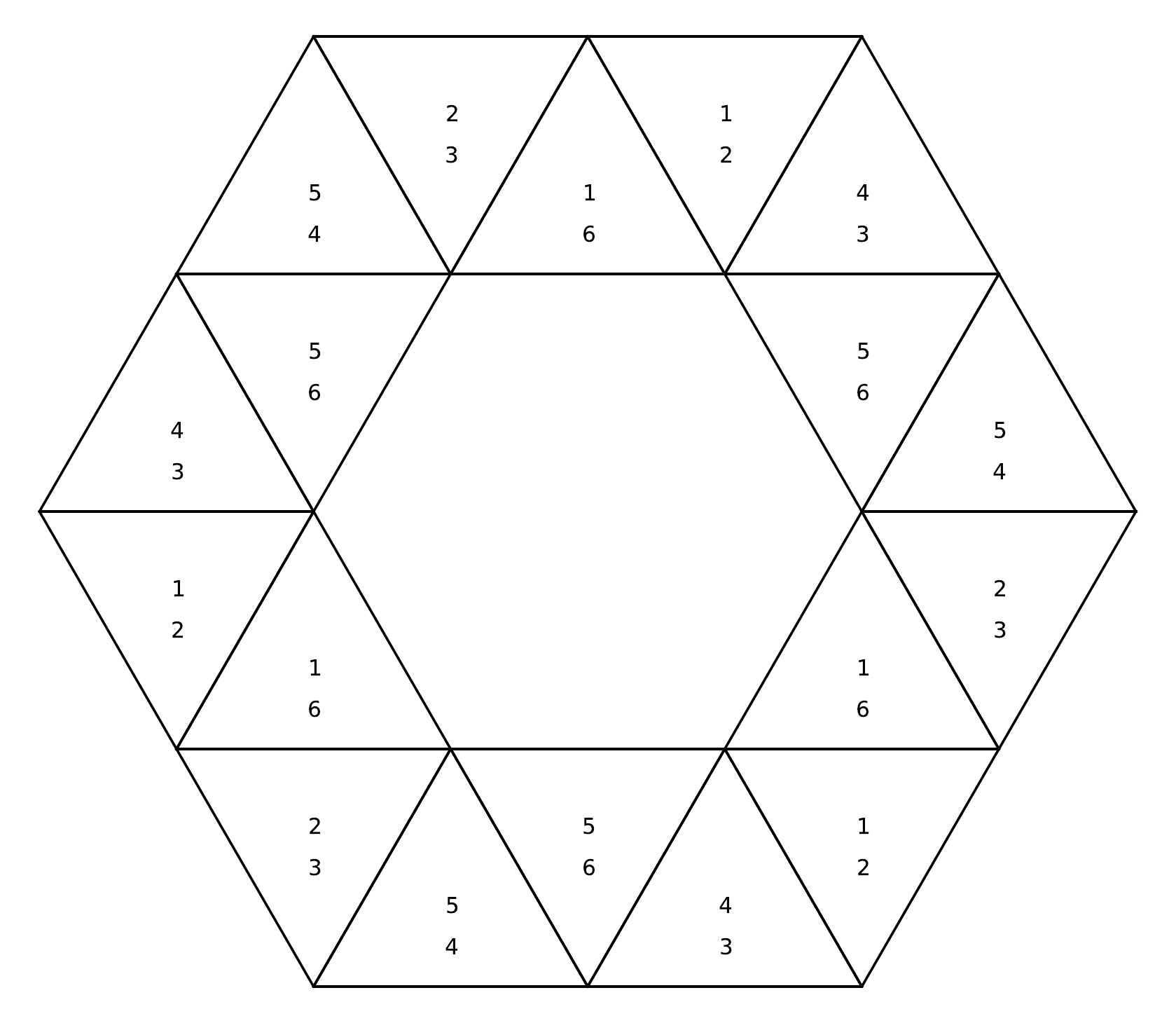}
\end{center}

Finally the third hexahexaflexagon has a sign sequence of $(1,1,-1,1,-1,-1)$:

\begin{center}
\includegraphics[width=3in]{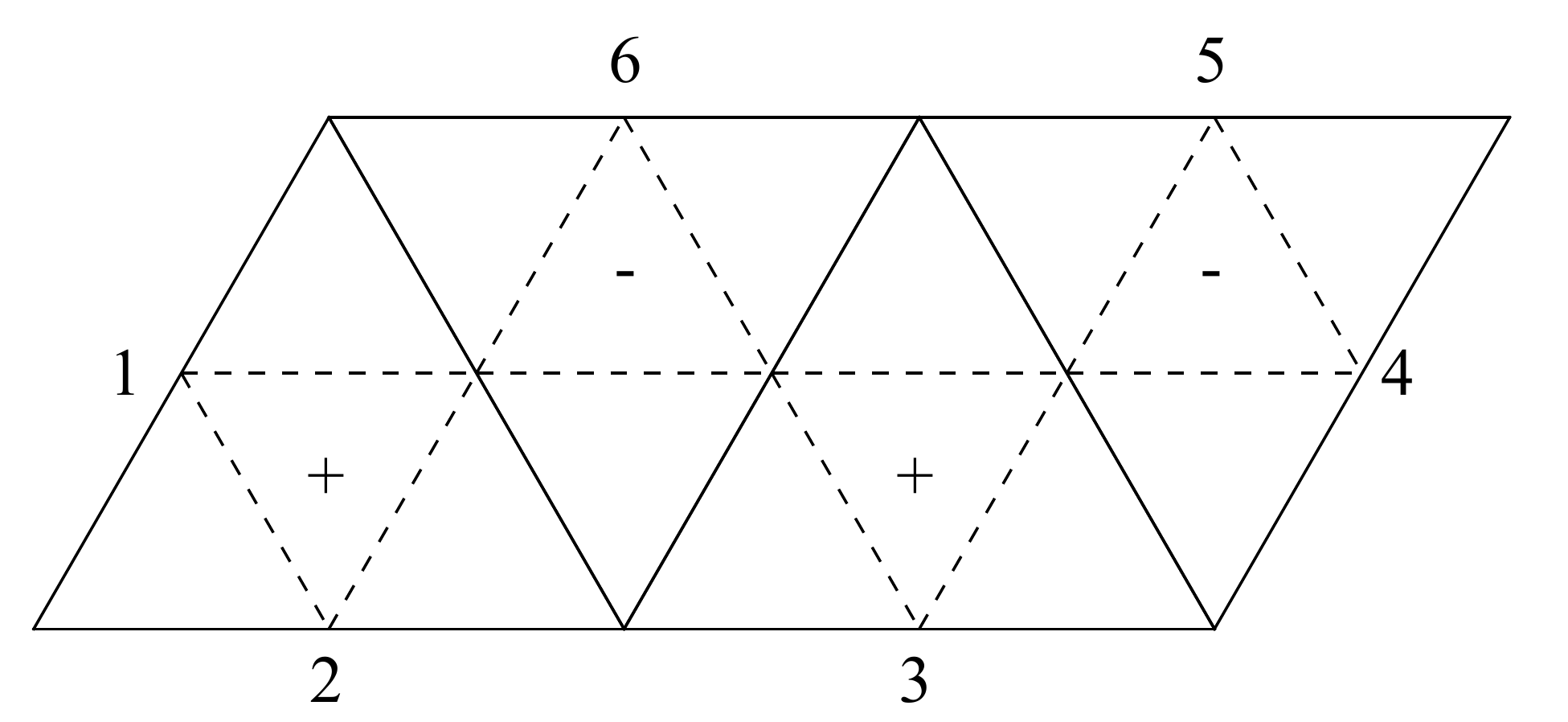}
\end{center}

\begin{center}
\includegraphics[width=3in]{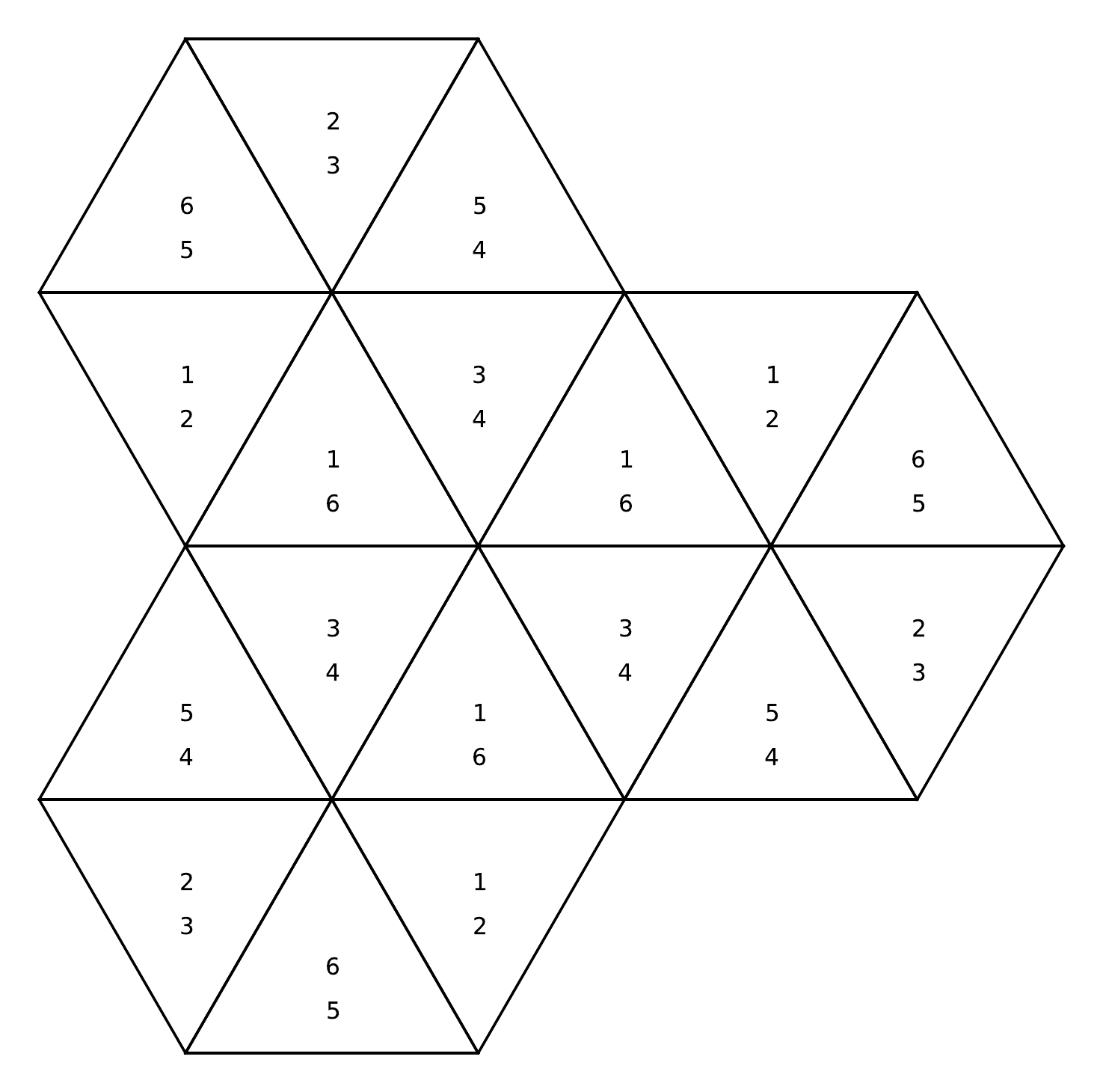}
\end{center}

\section{Counting hexaflexagons}

There are many ways one can count hexaflexagons and their foldings \cite{Maunsell1954, OakleyWisner, HPW1997, VflexIacob}.  Here we focus on the possible shapes of unfolded hexaflexagons, considering as equivalent cyclic shifts, reversals, and mirror images.  For a sign sequence $(a_1, a_2, \ldots, a_n)$, with $a_i \in \{1,-1\}$, this corresponds to the equivalence relations

$$\text{ cyclic shifts: }  (a_1, a_2, \ldots, a_n)  \sim (a_{i}, a_{i+1}, \ldots, a_n, a_1, a_2, \ldots a_{i-1}) $$
$$\text{ reversal: } ( a_1, a_2, \ldots, a_n)  \sim (a_n, a_{n-1}, \ldots, a_1) $$
$$\text{ inversion: }  (a_1, a_2, \ldots, a_n)  \sim (-a_1, -a_2, \ldots, -a_n) $$

We will adopt the convention of using representatives of equivalence classes of sign sequences which have a non-negative sum (i.e. $\sum_{i=1}^n a_i \ge 0$).

When we extend a hexaflexagon pattern with a new face, it is equivalent to replacing an element $a_i$ in the sign sequence by a pair of opposite signs $(-a_i,-a_i)$.  This means that every sign sequence will have at least one pair of equal adjacent signs.  

Since the initial trihexaflexagon has a representative sign sequence $(1,1,1)$ which sums to $3$, every extension of it to a four-face hexaflexagon will have sum $0$.  Further extensions will change the sum by $\pm 3$, so a sign sequence of length $n$ will have a sum in the set 

$$S(n) = \left \{ \begin{array}{l} \{-n, -n-6, \ldots n-6, n  \} \ \ \ \text{if $n \equiv 0$ mod 3}\\
\{-n+4, -n+10, \ldots n-10, n-4  \} \ \ \ \text{if $n \equiv 1$ mod 3}  \\
\{-n+2, -n+8, \ldots n-8, n -2 \} \ \ \ \text{if $n \equiv 2$ mod 3} 
\end{array} \right.$$

Remarkably these are the only constraints on the possible sign sequences:

\begin{lemma}
The set of all representatives of equivalence classes of sign sequences of hexaflexagons of length $n$ obtainable from extending the trihexaflexagon consists of all sequences with at least one pair of equal adjacent signs and with sum $\sum_{i=1}^n a_i \in S(n)$.
\end{lemma}
\begin{proof}
Our remarks preceding the lemma show that the given constraints are necessary.  We can complete the proof with induction. We know the lemma is true for $n<7$ by our previous constructions.  Consider a sequence of length $n>4$ that satisfies the constraints.  Under the cyclic shift equivalence, we can assume without loss of generality that $a_0 = a_1$ and that either $a_n = -a_0$ (if there are any changes of sign in the sequence) or $a_2 = a_3$.  In either case, the sequence of length $n-1$ obtained by replacing the pair $a_0,a_1$ by $-a_0$ satisfies the constraints.
\end{proof}

Fortunately we can make use of the extensive previous work done on enumerating binary necklaces and bracelets to count hexaflexagon equivalence classes.  A binary necklace is an equivalence class of sequences of two symbols, where cyclicly shifted sequences are considered equivalent.  A binary bracelet is an equivalence class of necklaces considered equivalent under reversal.  Our sign sequences of length $n$ with $k$ 1s and $n-k$ -1s correspond to bracelets and necklaces with $n$ beads, $k$ of which are black and $n-k$ of which are white.

The number of binary necklaces of length $n$ with $k$ ones is equal to (cf. \cite{KR}):
$$N(n,k) = \frac{1}{n} \sum_{j | gcd(n,k)} \phi(j) \left (  \begin{array}{c} n/j \\ k/j \end{array} \right )$$
where $\phi(j)$ is the Euler totient function and $\displaystyle \left (  \begin{array}{c} n/j \\ k/j \end{array} \right )$ is a binomial coefficient.

The number of binary bracelets of length $n$ with $k$ ones is also known \cite{Karim}:
$$B(n,k) = \frac{1}{2}N(n,k) + \left \{ \begin{array}{rr}  \frac{1}{2} \left (  \begin{array}{c} (n-1)/2 \\ (k-1)/2 \end{array} \right ) & \text{ if $n$ and $k$ are odd }\\
& \\
\frac{1}{2} \left (  \begin{array}{c} (n-1)/2 \\ k/2\end{array} \right ) & \text{ if $n$ is odd and $k$ is even }\\
& \\
\frac{1}{4} \left (  \begin{array}{c} n/2 - 1\\ k/2\end{array} \right ) + \frac{1}{4} \left (  \begin{array}{c} n/2 \\ k/2 -1\end{array} \right ) &  \text{ if $n$ and $k$ are even } \\
& \\
\frac{1}{2} \left (  \begin{array}{c} n/2 - 1 \\ (k-1)/2\end{array} \right ) & \text{ if $n$ is even and $k$ is odd }\\
 \end{array} \right .$$
 
We will also need the well-known (cf. \cite{Moree}) formula for the number of aperiodic necklaces (also known as Lyndon words) of $n$ beads with $k$ black beads:
 
 $$L(n,k) = \frac{1}{n} \sum_{j | gcd(n,k)} \mu(j) \left (  \begin{array}{c} n/j \\ k/j \end{array} \right )$$
 where $\mu(j)$ is the M\"{o}bius function (equal to the sum of the primitive $j$th roots of unity).
 
 Armed with these formulae we can prove the following:
 
 \begin{theorem}
 The number of hexaflexagon equivalence classes with $n$ faces is 
 
 $$H(n) = \left \{ \begin{array}{c}
 \sum_{i=0}^{\lfloor (n-2)/6 \rfloor}  B(n, \lceil n/2 \rceil + 1 + 3i) \ \ \text{ for $n$ odd} \\
-  B(n,n/2)/2 + F(n)/2 - 1 + \sum_{i=0}^{\lfloor n/6 \rfloor}  B(n, \lceil n/2 \rceil  + 3i) \ \  \text{ for $n$ even}  
\end{array} \right . $$
where
$$F(n) = \sum_{k=1}^{n/2} \sum_{L | gcd(n/2,k)} L(\frac{n}{2 l},\frac{k}{l}) \lceil \frac{k}{2 l} \rceil .$$
 \end{theorem}
 \begin{proof}

For odd $n$ we can assume that there are more $1$s than $-1$s in the sign sequence, and the number of distinct hexaflexagon patterns with $n$ faces is
$$\sum_{i=0}^{\lfloor (n-2)/6 \rfloor}  B(n, \lceil n/2 \rceil + 1 + 3i) $$
since we can generate any bracelet of the length $n$ with a total of $\lceil n/2 \rceil + 1 + 3i$ 1s in its sign sequence for $i \in (0, \lfloor (n-2)/6 \rfloor)$. 


For even $n$ the counting is a little more complicated since there are some sign sequences which are equivalent to each other under the interchange of $1$s and $-1$s (the inversion equivalence).  There is also one sign sequence for even $n$ which cannot be obtained from the trihexaflexagon, the alternating sequence $(1,-1,1,-1,\ldots,1,-1)$, since this does not contain a pair of equal adjacent entries.  Taking these two issues into account we obtain for even $n$:

$$ H(n)  = -  B(n,n/2)/2 + F(n)/2 - 1 + \sum_{i=0}^{\lfloor n/6 \rfloor }  B(n, \lceil n/2 \rceil  + 3i) $$
 where $F(n)$ is the number of binary bracelets with equal numbers of positive and negative entries which are fixed under inversion; we will call these self-conjugate.  
 
 
The number of self-conjugate sign sequences of length $n$, $F(n)$, can be computed in terms of $L(n,k)$, using a bijective correspondence noted in \cite{KR}: the number of necklaces of $n$ beads with $k$ black and $n-k$ white, which is our $N(n,k)$, is equal to the number of cyclic compositions of $n$ into $k$ parts.  A self-conjugate sequence can be thought of as two identical interwoven necklaces of length $n/2$, each with $k$ parts.  To count the distinct interweavings we decompose the number of such necklaces by their minimal periodicity, yielding

 $$F(n) = \sum_{k=1}^{n/2} \sum_{L | gcd(n/2,k)} L(\frac{n}{2 l},\frac{k}{l}) \lceil \frac{k}{2 l} \rceil$$
 
 The factor $\lceil \frac{k}{2 l} \rceil$ is the number of ways of interweaving the two identical necklaces.
 
 
 \end{proof}
 
 In the following table we have computed the number of distinct hexaflexagons $H(n)$ for each $n$ from $3$ to $26$, as well as the number $H_p(n)$ of hexaflexagons which are printable in the sense that the unfolded hexaflexagon does not overlap itself.  
 
\vspace{.2in}
\begin{center}
\begin{tabular}{|c|c|c|c|c|c|c|} \hline
$N$ & $H(N)$ & $H_p(N)$ & \hspace{.2in} &  N & $H(N)$ & $H_p(N)$\\ \hline
3 & 1 & 1 & & 15 & 205 & 78 \\ \hline
4 & 1 & 1 & &  16 & 411 & 144 \\ \hline
5 & 1 & 1 & &  17 & 685 & 224 \\ \hline
6 & 3 & 3  & &  18 & 1353 & 421 \\ \hline
7 & 3 & 2  & &  19 & 2385 & 648 \\ \hline
8 & 7 & 5  & &  20 & 4643 & 1185 \\ \hline
9 & 8 & 6  & &  21  & 8496 & 1990 \\ \hline
10 & 17 & 10  & &  22  & 16430 & 3668 \\ \hline
11 & 21 & 11  & &  23  & 30735 & 6095 \\ \hline
12 & 47 & 21  & &  24  & 59343 & 11079 \\ \hline
13 & 63 & 29  & &  25  & 112531 & 19098 \\ \hline
14 & 132 & 58  & &  26 & 217245 & 34891 \\ \hline
\end{tabular}
\end{center}
\vspace{.2in}

The first example of a non-printable hexaflexagon is for $n=7$, with the representative sign sequence $(1, 1, 1, 1, -1, 1, -1)$.  This sign sequence corresponds to the path shown below:

\begin{center} \begin{figure}[h!t]
\includegraphics[width=2.5in]{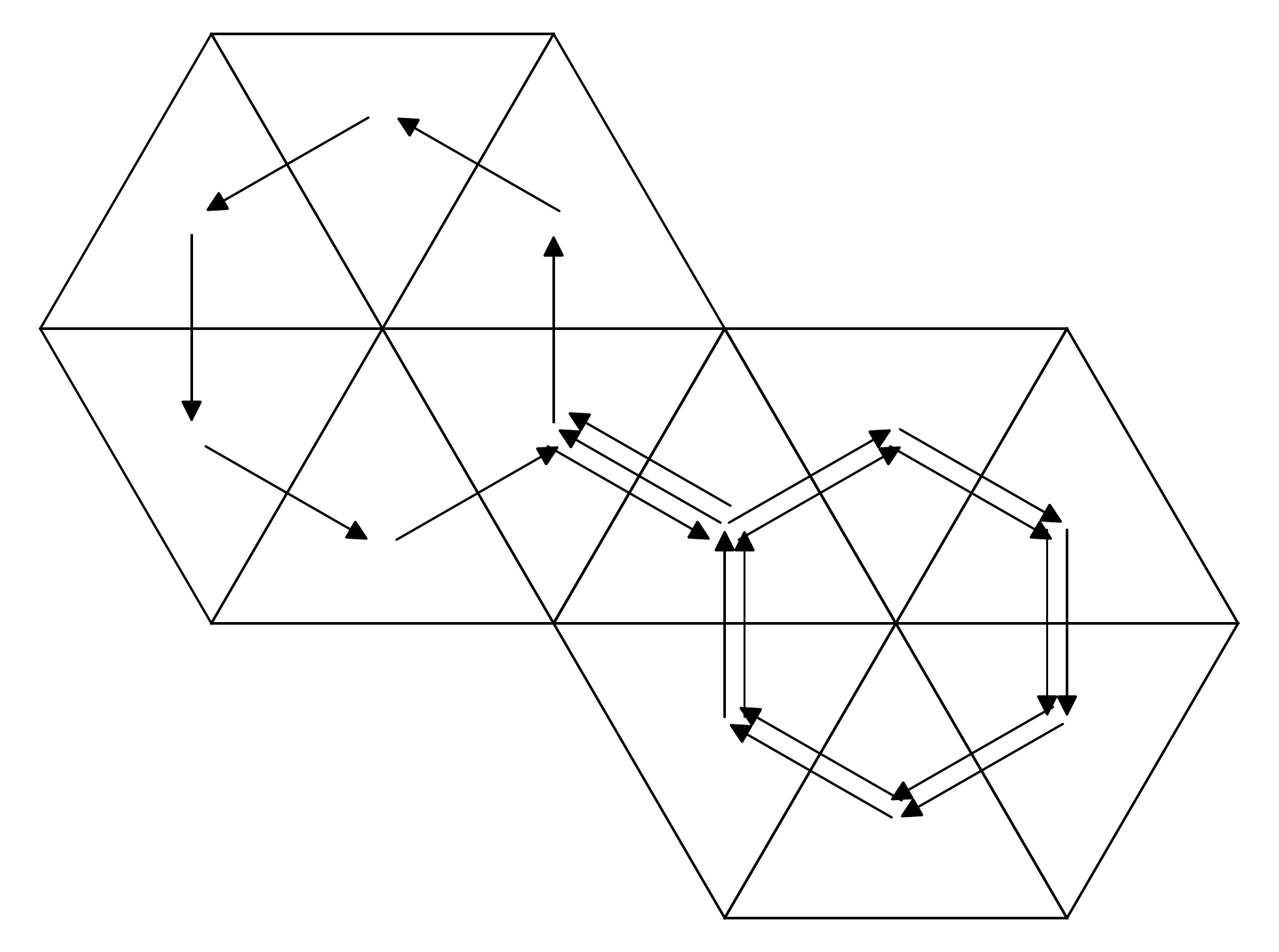}
 \end{figure} \end{center} 

We do not presently have a formula for $H_p$, and it seems that this would be quite difficult to obtain explicitly. 

\section{Acknowledgements}

The author thanks Vi Hart for her wonderful videos, which introduced him to this subject, and Richard J. Mathar for finding errors in a previous version of this article.

\end{document}